\documentclass[a4paper, 10pt, conference]{ieeeconf}
\IEEEoverridecommandlockouts
\usepackage{amsmath}
\usepackage{amssymb}
\usepackage{enumerate}
\usepackage{mathrsfs}
\usepackage{cases}
\usepackage[table]{xcolor} 
\usepackage{tikz}
\usepackage{makecell}
\usepackage{multirow}
\usepackage{algpseudocode}
\usepackage{algorithmicx,algorithm}
\usepackage{graphicx}
\definecolor{darkblue}{rgb}{0.0, 0.0, 0.55}
\definecolor{bordeaux}{rgb}{0.34, 0.01, 0.1}
\usepackage[pagebackref,colorlinks,linkcolor=bordeaux,citecolor=darkblue,urlcolor=black,hypertexnames=true]{hyperref}
\usetikzlibrary{arrows,positioning}
\usepackage{siunitx}

\newtheorem{theorem}{Theorem}[section]

\newtheorem{definition}[theorem]{Definition}
\newtheorem{example}[theorem]{Example}

\newtheorem{proposition}[theorem]{Proposition}

\newtheorem{remark}[theorem]{Remark}

\def\R{{\mathbb{R}}}

\def\N{{\mathbb{N}}}

\def\x{{\mathbf{x}}}

\def\a{{\boldsymbol{\alpha}}}

\def\b{{\boldsymbol{\beta}}}
\def\g{{\boldsymbol{\gamma}}}

\def\A{{\mathscr{A}}}
\def\AA{{\mathcal{A}}}
\def\B{{\mathscr{B}}}

\def\supp{\hbox{\rm{supp}}}

\def\int{\hbox{\rm{int}}}
\def\New{\hbox{\rm{New}}}

\newif\ifcomment
\commentfalse
\commenttrue

\makeatletter
\g@addto@macro\normalsize{%
 \setlength\abovedisplayskip{4pt plus 5pt minus 0pt} 
 \setlength\belowdisplayskip{4pt plus 5pt minus 0pt} 
 \setlength\abovedisplayshortskip{4pt plus 5pt minus 0pt} 
 \setlength\belowdisplayshortskip{4pt plus 5pt minus 0pt} 
}
\makeatother

\title{SparseJSR: A Fast Algorithm to Compute Joint Spectral Radius via Sparse SOS Decompositions}
\author{Jie Wang, Martina Maggio and Victor Magron}

\begin{document}

\maketitle
\thispagestyle{empty}
\pagestyle{empty}

\begin{abstract}
This paper focuses on the computation of the joint spectral radius (JSR), when the involved matrices are sparse.
We provide a sparse variant of the procedure proposed by Parrilo and Jadbabaie to compute upper bounds of the JSR by means of sum-of-squares (SOS) programming.
Our resulting iterative algorithm, called {\tt SparseJSR}, is based on the \emph{term sparsity} SOS (TSSOS) framework developed by Wang, Magron and Lasserre, which yields SOS decompositions of polynomials with arbitrary sparse supports. 
{\tt SparseJSR} exploits the sparsity of the input matrices to significantly reduce the computational burden associated with the JSR computation. 
Our algorithmic framework is then successfully applied to compute upper bounds for JSR on randomly generated benchmarks as well as on problems arising from stability proofs of controllers, in relation with possible hardware and software faults.
\end{abstract}

\maketitle

\section{Introduction}
Given a set of matrices $\AA=\{A_1,\ldots,A_m\}\subseteq\R^{n\times n}$, the {\em joint spectral radius} (JSR) of $\AA$ is defined by 
\begin{equation}\label{jsr}
\rho(\AA):=\lim_{k\rightarrow\infty}\max_{\sigma\in\{1,\ldots,m\}^k}||A_{\sigma_1}A_{\sigma_2}\cdots A_{\sigma_k}||^{\frac{1}{k}},
\end{equation}
which characterizes the maximal asymptotic growth rate of products of matrices from $\AA$. Note that the value of $\rho(\AA)$ is independent of the choice of the norm used in \eqref{jsr}. When $\AA$ contains a single matrix, the JSR coincides with the usual spectral radius. Hence JSR can be viewed as a generalization of the usual spectral radius to the case of multiple matrices.

The concept of JSR was first introduced by Rota and Strang in \cite{rota} and since then has found applications in many areas such as the stability of switched linear dynamical systems, the continuity of wavelet functions, combinatorics and language theory, the capacity of some codes, the trackability of graphs. We refer the readers to \cite{jungers} for a survey of the theory and applications of JSR.

Inspired by the various applications, there has been a lot of work on the computation of JSR; see e.g.  \cite{ahmadi,blondel3,gripenberg,guglielmi1,parrilo,protasov} to name a few. Unfortunately, it turns out that the exact computation and even the approximation of JSR are notoriously difficult \cite{Blondel2}. It was proved in \cite{blondel1} that the problem of deciding whether $\rho(\AA)\le1$ is undecidable even for $\AA$ consisting of two matrices. Therefore, various methods focus on computing lower bounds and upper bounds for JSR \cite{ahmadi,blondel3,gripenberg,parrilo}.

Parrilo and Jadbabaie proposed in \cite{parrilo} a sum-of-squares (SOS) approach which makes use of semidefinite programming (SDP) to compute a sequence of upper bounds $\{\rho_{SOS,2d}(\AA)\}_{d\ge1}$ for $\rho(\AA)$. They proved that the sequence $\{\rho_{SOS,2d}(\AA)\}_{d\ge1}$ converges to $\rho(\AA)$ when $d$ increases. In practice, mostly often even small $d$ (e.g., $d=1,2$) can provide upper bounds of good quality for $\rho(\AA)$. Once the upper bound coincides with a lower bound provided by other methods, then we obtain the exact value of the JSR. However, the computational burden of the SOS approach grows rapidly when the matrix size or $d$ increases. Given the current state of SDP solvers, this approach can only handle matrices of modest sizes when $d\ge2$.

For general polynomial optimization problems (POP), one way to reduce the computational cost of the associated SOS relaxations is to exploit the so-called \emph{correlative sparsity pattern} relative to the variables of the POP \cite{waki}.  
To build these sparse SOS relaxations, one relies on the \emph{correlative sparsity pattern (csp) graph} of the POP. The nodes of the csp graph are the variables and two nodes are connected via an edge when the corresponding variables appear in the same term of the objective function or in the same constraint involved in the POP.
This approach was successfully used for several interesting applications, including certified roundoff error bounds \cite{toms17}, optimal powerflow problems \cite{josz2018lasserre}, noncommutative optimization \cite{klep2019sparse}, Lipschitz constants of ReLU networks \cite{chen2020polynomial}, robust geometric perception \cite{yang2020one}.

A complementary workaround is to take into account 
\emph{term sparsity} (TS) of the input data to obtain sparse SOS relaxations, as recently studied in \cite{wang,wang2,wang3}, yielding the so-called TSSOS framework.
%
TSSOS relies on the {\em term sparsity pattern (tsp) graph} related to the input polynomials. 
To build the associated sparse SOS relaxations, one connects the nodes of this graph (corresponding to monomials from a monomial basis) whenever the product of the corresponding monomials either appears in the supports of input polynomials or is a monomial of even degree.
Recent applications include learning and forecasting of linear systems  \cite{zhou2020proper,zhou2020fairness} via reformulation into noncommutative polynomial optimization and exploiting term sparsity to reduce the size of the associated relaxations.
Note that term sparsity can be combined with correlative sparsity to reduce even further the size of the associated relaxations \cite{miller2019decomposed,wang4}.

The original underlying motivation of this paper was to apply term sparsity to improve the scalability of JSR computation arising from the study of deadline hit and deadline miss \cite{maggio2020control}. In this case, the computation of the control signal can fail due to a hardware and software fault, causing either no update or a delayed application of the control signal. The main application in this case is to determine how long the controller can operate in a faulty state (in which it does not complete the computation in due time, causing a deadline miss) before the stability of the system is compromised.
The idea is to bound the JSR of products between state matrices associated to deadline hit and deadline miss by solving a POP \cite{ahmadi}.
For such JSR problems, matrices of large sizes issued from applications reveal certain kinds of sparsity in many cases.
A natural question is: can we exploit the sparsity of matrices to improve the scalability of the SOS approach and to compute upper bounds more efficiently? In this paper, we address this specific question. 

\vspace{1em}
\noindent\textbf{Contributions and outline}:
In Section \ref{sec:background}, we recall preliminary background about SOS forms, chordal graphs and approximation of JSR via SOS programming. 
To make the current paper as self-contained as possible, Section \ref{sec:sparsesos} is dedicated to detailed explanation about sparse SOS decompositions via generation of smaller monomial bases and exploitation of the block structure of Gram matrices.
Our main contribution is described in Section \ref{sec:sparsejsr}. We propose a so-called {\tt SparseJSR} algorithm, which is based on the SOS approach and in coordination with the sparsity of matrices appearing within the JSR computation. The algorithm is implemented in the open-source Julia package, also called {\tt SparseJSR}, and is freely available\footnote{%
\href{https://github.com/wangjie212/SparseJSR}{https://github.com/wangjie212/SparseJSR}}.
The performance of {\tt SparseJSR} is then illustrated in Section \ref{sec:benchs}, first on randomly generated benchmarks, and then on benchmarks coming from the study of deadline hit/miss in \cite{maggio2020control}.
Although our sparse version of the SOS approach is not guaranteed to produce upper bounds for JSR as good as the dense one with the same relaxation order, the numerical experiments in this paper demonstrate that our sparse approach is able to produce upper bounds of rather good quality but at a significantly cheaper computational cost compared to the dense approach.

\section{Notation and Preliminaries}
\label{sec:background}
Let $\x=(x_1,\ldots,x_n)$ be a tuple of variables and $\R[\x]=\R[x_1,\ldots,x_n]$ be the ring of real $n$-variate polynomials. We use $\R[\x]_{2d}$ to denote the set of forms (i.e., homogeneous polynomials) of degree $2d$ for $d\in\N$. A polynomial $f\in\R[\x]$ can be written as $f(\x)=\sum_{\a\in\A}f_{\a}\x^{\a}$ with $f_{\a}\in\R, \x^{\a}=x_1^{\alpha_1}\cdots x_n^{\alpha_n}$. The {\em support} of $f$ is defined by $\supp(f):=\{\a\in\A\mid f_{\a}\ne0\}$. We use $|\cdot|$ to denote the cardinality of a set. For a nonempty finite set $\A\subseteq\N^n$, let $\R[\A]$ be the set of polynomials in $\R[\x]$ whose supports are contained in $\A$, i.e., $\R[\A]=\{f\in\R[\x]\mid\supp(f)\subseteq\A\}$ and let $\x^{\A}$ be the $|\A|$-dimensional column vector consisting of elements $\x^{\a},\a\in\A$ (fix any ordering on $\N^n$). For convenience, we abuse notation a bit in this paper and use also $\B\subseteq\N^n$ (resp. $\b\in\N^n$) to denote a monomial set (resp. a monomial). For a positive integer $r$, let $\mathbf{S}^r$ be the set of $r\times r$ symmetric matrices and the set of $r\times r$ positive semidefinite (PSD) matrices is denoted by $\mathbf{S}_+^r$.

\subsection{SOS forms}
Given a form $f\in\R[\x]_{2d}$ with $d\in\N$, if there exist forms $f_1,\ldots,f_t\in\R[\x]_{d}$ such that $f=\sum_{i=1}^tf_i^2$, then we say that $f$ is a \textit{sum-of-squares} (SOS) form. The set of SOS forms in $\R[\x]_{2d}$ is denoted by $\Sigma_{n,2d}$. For $d\in\N$, let $\N^n_d:=\{(\alpha_i)_{i=1}^n\in\N^n\mid\sum_{i=1}^n\alpha_i=d\}$ and assume that $f\in\R[\x]_{2d}$. Then deciding whether $f\in\Sigma_{n,2d}$ is equivalent to verifying the existence of a PSD matrix $Q$ (which is called a \textit{Gram matrix} for $f$) such that
\begin{equation}\label{gram}
f=(\x^{\N^n_{d}})^TQ\x^{\N^n_{d}},
\end{equation}
which can be formulated as a semidefinite program (SDP). The monomial basis $\x^{\N^n_{d}}$ used in \eqref{gram} is called the {\em standard monomial basis}.

\subsection{Chordal graphs and sparse matrices}
An (undirected) {\em graph} $G(V,E)$ or simply $G$ consists of a set of nodes $V$ and a set of edges $E\subseteq\{\{v_i,v_j\}\mid (v_i,v_j)\in V\times V\}$. For a graph $G(V,E)$, a {\em cycle} of length $k$ is a set of nodes $\{v_1,v_2,\ldots,v_k\}\subseteq V$ with $\{v_k,v_1\}\in E$ and $\{v_i, v_{i+1}\}\in E$ for $i=1,\ldots,k-1$. A {\em chord} in a cycle $\{v_1,v_2,\ldots,v_k\}$ is an edge $\{v_i, v_j\}$ that joins two nonconsecutive nodes in the cycle. A graph is called a {\em chordal graph} if all its cycles of length at least four have a chord. Chordal graphs include some common classes of graphs, such as complete graphs, line graphs and trees, and have applications in sparse matrix theory \cite{va}. Any non-chordal graph $G(V,E)$ can always be extended to a chordal graph $\overline{G}(V,\overline{E})$ by adding appropriate edges to $E$, which is called a {\em chordal extension} of $G(V,E)$. A {\em clique} $C\subseteq V$ of $G$ is a subset of nodes where $\{v_i,v_j\}\in E$ for any $v_i,v_j\in C$. If a clique $C$ is not a subset of any other clique, then it is called a {\em maximal clique}. It is known that maximal cliques of a chordal graph can be enumerated efficiently in linear time in the number of nodes and edges of the graph \cite{bp}.

For a graph $G$, the chordal extension of $G$ is usually not unique. We would prefer a chordal extension with the smallest clique number. Finding a chordal extension with the smallest clique number is an NP-complete problem in general. Fortunately, several heuristic algorithms are known to efficiently produce a good approximation \cite{treewidth}.

Given a graph $G(V,E)$, a symmetric matrix $Q$ with row and column indices labeled by $V$ is said to have sparsity pattern $G$ if $Q_{\b\g}=Q_{\g\b}=0$ whenever $\b\ne\g$ and $\{\b,\g\}\notin E$. Let $\mathbf{S}_G$ be the set of symmetric matrices with sparsity pattern $G$. A matrix in $\mathbf{S}_G$ exhibits a block structure (after an appropriate permutation of rows and columns) as illustrated in Figure \ref{qbd}. Each block corresponds to a maximal clique of $G$. The maximal block size is the maximal size of maximal cliques of $G$, namely, the \emph{clique number} of $G$. Note that there might be overlaps between blocks because different maximal cliques may share nodes.

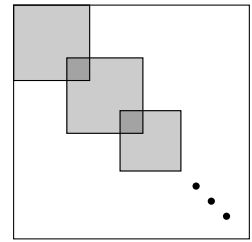
\begin{figure}[htbp]
\begin{center}
\begin{tikzpicture}
\draw (0,0) rectangle (3.1,-3.1);
\draw[fill=black, opacity=0.2] (0,0) rectangle (1,-1);
\draw (0,0) rectangle (1,-1);
\draw[fill=black, opacity=0.2] (0.7,-0.7) rectangle (1.7,-1.7);
\draw (0.7,-0.7) rectangle (1.7,-1.7);
\draw[fill=black, opacity=0.2] (1.4,-1.4) rectangle (2.2,-2.2);
\draw (1.4,-1.4) rectangle (2.2,-2.2);
\fill (2.4,-2.4) circle (0.3ex);
\fill (2.6,-2.6) circle (0.3ex);
\fill (2.8,-2.8) circle (0.3ex);
\end{tikzpicture}
\end{center}
\caption{A block structure of matrices in $\mathbf{S}_G$. The gray area indicates the positions of possible nonzero entries.}\label{qbd}
\end{figure}

Given a maximal clique $C$ of $G(V,E)$, we define an {\em indexing matrix} $P_{C}\in \R^{|C|\times |V|}$ as
\begin{equation}\label{sec2-eq6}
[P_{C}]_{i\b}=\begin{cases}
1, &\textrm{if }C(i)=\b,\\
0, &\textrm{otherwise},
\end{cases}
\end{equation}
where $C(i)$ denotes the $i$-th node in $C$, sorted in the ordering compatible with $V$. Note that $Q_{C}=P_{C}QP_{C}^T\in \mathbf{S}^{|C|}$ extracts a principal submatrix $Q_C$ defined by the indices in the clique $C$ from a symmetric matrix $Q$, and $Q=P_{C}^TQ_{C}P_{C}$ inflates a $|C|\times|C|$ matrix $Q_{C}$ into a sparse $|V|\times |V|$ matrix $Q$.

PSD matrices with sparsity pattern $G$ form a convex cone
\begin{equation}\label{sec2-eq5}
\mathbf{S}_+^{|V|}\cap\mathbf{S}_G=\{Q\in\mathbf{S}_G\mid Q\succeq0\}.
\end{equation}
When the sparsity pattern graph $G$ is chordal, the cone $\mathbf{S}_+^{|V|}\cap\mathbf{S}_G$ can be
decomposed as a sum of simple convex cones, as stated in the following theorem.
\begin{theorem}[\cite{agler}]\label{sec2-thm}
Let $G(V,E)$ be a chordal graph and assume that $C_1,\ldots,C_t$ are the list of maximal cliques of $G(V,E)$. Then a matrix $Q\in\mathbf{S}_+^{|V|}\cap\mathbf{S}_G$ if and only if there exists $Q_{k}\in \mathbf{S}_+^{|C_k|}$ for $k=1,\ldots,t$ such that $Q=\sum_{k=1}^tP_{C_k}^TQ_{k}P_{C_k}$.
\end{theorem}

For more details about sparse matrices and chordal graphs, the reader may refer to \cite{va}.


\subsection{Approximating the joint spectral radius via SOS relaxations}
The joint spectral radius (JSR) for a set of matrices $\AA=\{A_1,\ldots,A_m\}\subseteq\R^{n\times n}$ is given by 
\begin{equation}
\rho(\AA):=\lim_{k\rightarrow\infty}\max_{\sigma\in\{1,\ldots,m\}^k}||A_{\sigma_1}A_{\sigma_2}\cdots A_{\sigma_k}||^{\frac{1}{k}} \,.
\end{equation}
Parrilo and Jadbabaie proposed to compute a sequence of upper bounds for $\rho(\AA)$ via SOS relaxations. The core idea is based on the following theorem.
\begin{theorem}[\cite{parrilo}, Theorem 2.2]\label{sec2-thm1}
Given a set of matrices $\AA=\{A_1,\ldots,A_m\}\subseteq\R^{n\times n}$, let $p$ be a strictly positive form of degree $2d$ that satisfies
\begin{equation*}
    p(A_i\x)\le\gamma^{2d}p(\x),\quad\forall\x\in\R^n,\quad i=1,\ldots,m.
\end{equation*}
Then, $\rho(\AA)\le\gamma$.
\end{theorem}

Replacing positive forms by more tractable SOS forms, Theorem \ref{sec2-thm1} immediately suggests the following SOS relaxations indexed by $d\in\N\backslash\{0\}$ to compute a sequence of upper bounds for $\rho(\AA)$: 
\begin{align}\label{densesos}
    \rho_{\textrm{SOS},2d}(\AA):=&\inf_{p\in\R[\x]_{2d},\gamma}  \gamma \\
    & \textrm{ s.t. } 
    \begin{cases}p(\x)-||\x||_2^{2d}\in\Sigma_{n,2d},\\
    \gamma^{2d}p(\x)-p(A_i\x)\in\Sigma_{n,2d},\ 1\leq i \leq m.
    \end{cases} \nonumber
\end{align}

The terms ``$||\x||_2^{2d}$" is added to make sure $p$ is strictly positive.
The optimization problem \eqref{densesos} can be solved via SDP by bisection on $\gamma$. It was shown in \cite{parrilo} that the upper bound $\rho_{\textrm{SOS},2d}(\AA)$ satisfies the following theorem.
\begin{theorem}[\cite{parrilo}]\label{sec2-thm2}
Let $\AA=\{A_1,\ldots,A_m\}\subseteq\R^{n\times n}$. For any integer $d\ge1$, one has $m^{-\frac{1}{2d}}\rho_{\textrm{SOS},2d}(\AA)\le\rho(\AA)\le\rho_{\textrm{SOS},2d}(\AA)$.
\end{theorem}
It is immediate from Theorem \ref{sec2-thm2} that $\{\rho_{\textrm{SOS},2d}(\AA)\}_{d\ge1}$ converges to $\rho(\AA)$ when $d$ increases.

\section{Sparse SOS Decompositions}
\label{sec:sparsesos}
Deciding whether a form $f$ is SOS involves solving an SDP whose size scales combinatorially with the number of variables and the degree of $f$. When $f$ is sparse, it is possible to exploit the sparsity to construct an SDP of smaller size in order to reduce the computational burden. This includes two aspects: generating a smaller monomial basis and exploiting block structures for Gram matrices.

\subsection{Generating a smaller monomial basis}
Given a polynomial $f\in\R[\x]$, the \emph{Newton polytope} of $f$ is the convex hull of the support of $f$.
It is known that the standard monomial basis $\N^n_{d}$ used in \eqref{gram} can be replaced by the integer points in half of the Newton polytope of $f$, i.e., by
\begin{equation}\label{sec3-eq1}
\B=\frac{1}{2}\New(f)\cap\N^n\subseteq\N^n_{d}.
\end{equation}
See, e.g., \cite{re} for a proof.

In \cite{wang3}, an algorithm named ${\tt GenerateBasis}$ was proposed to generate a smaller monomial basis for \eqref{gram} than the one provided by the Newton polytope.
Given the support of $f$, the output of ${\tt GenerateBasis}$ is an increasing chain of monomial sets:
$$\B_1\subseteq\B_2\subseteq\B_3\subseteq\cdots\subseteq\N^n_{d}.$$
Each $\B_p$ can serve as a candidate monomial basis. In practice, if indexing the unknown Gram matrix from \eqref{gram} by $\B_p$ leads to an infeasible SDP, then we turn to $\B_{p+1}$ until a feasible SDP is retrieved. In many cases, the algorithm ${\tt GenerateBasis}$ can provide a monomial basis smaller than the one given by \eqref{sec3-eq1}; see \cite{wang3} for such examples.
\begin{remark}
For all tested examples, $\B_1$ is a suitable monomial basis, but we do not know if this is true in general.
\end{remark}
\subsection{Term sparsity patterns}
To derive a block structure for Gram matrices, we recall the concept of term sparsity patterns \cite{wang,wang2,wang3}.
\begin{definition}
Let $f(\x)\in\R[\x]$ with $\supp(f)=\A$. Assume that $\B$ is a monomial basis. The {\em term sparsity pattern graph} $G(V,E)$ of $f$ is defined by $V=\B$ and
\begin{equation}\label{tsp}
    E=\{\{\b,\g\}\mid \b,\g \in V,\,\b\ne\g,\,\b+\g\in\A\cup2\B\},
\end{equation}
where $2\B=\{2\b\mid\b\in\B\}$.
\end{definition}
For a term sparsity pattern graph $G(V,E)$, we denote a chordal extension of $G$ by $\overline{G}(V,\overline{E})$.

\begin{example}\label{ex1}
Consider the polynomial $f=x_1^4+x_2^4+x_3^4+x_1x_2x_3^2+x_1x_2^2x_3$.
A monomial basis for $f$ is $\{x_1^2,x_2^2,x_3^2,x_1x_2,x_1x_3,x_2x_3\}$. See Figure \ref{chordal} for the term sparsity pattern graph $G$ of $f$ and a chordal extension $\overline{G}$ of $G$.
\begin{figure}[htbp]
\begin{center}
{\tiny
\begin{tikzpicture}[every node/.style={circle, draw=black!50, thick, minimum size=7.5mm}]
\node (n2) at (90:1.5) {$x_1^2$};
\node (n3) at (30:1.5) {$x_3^2$};
\node (n4) at (330:1.5) {$x_1x_2$};
\node (n5) at (270:1.5) {$x_2x_3$};
\node (n6) at (210:1.5) {$x_1x_3$};
\node (n1) at (150:1.5) {$x_2^2$};
\draw (n2)--(n3);
\draw (n1)--(n3);
\draw[dashed] (n3)--(n5);
\draw[dashed] (n3)--(n6);
\draw (n3)--(n4);
\draw (n4)--(n5);
\draw (n5)--(n6);
\draw (n6)--(n1);
\draw (n1)--(n2);
\end{tikzpicture}}\\
\end{center}
\caption{The term sparsity pattern graph and a chordal extension for Example \ref{ex1}. The dashed edges are added after a chordal extension.}\label{chordal}
\end{figure}
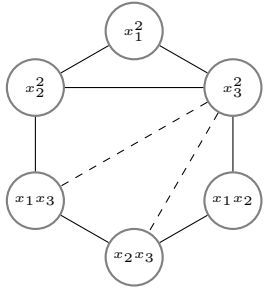 
\end{example}

Given a sparse SOS form $f(\x)\in\R[\A]$ and a monomial basis $\B$, generally a Gram matrix for $f$ is not necessarily sparse. Let $G$ be the term sparsity pattern graph of $f$ and $\overline{G}$ a chordal extension of $G$. To get a sparse SOS decomposition of $f$, we then impose the sparsity pattern $\overline{G}$ to the Gram matrix for $f$, i.e., we consider the following subset of SOS forms in  $\Sigma_{n,2d}$:
\begin{equation*}
\Sigma_{\A}:=\{f\in\R[\A]\mid\exists Q\in\mathbf{S}_+^{|\B|}\cap\mathbf{S}_{\overline{G}}\,\textrm{ s.t. }f=(\x^{\B})^TQ\x^{\B}\}.
\end{equation*}
Theorem \ref{sec2-thm} enables us to give the following sparse SOS decompositions for polynomials in $\Sigma_{\A}$.
\begin{theorem}[\cite{wang}, Theorem 3.3]\label{sec3-thm1}
Given $\A\subseteq\N^n$, assume that $\B=\{\b_1,\ldots,\b_r\}$ is a monomial basis and $G$ is the term sparsity pattern graph. Let $C_1, C_2, \ldots, C_t\subseteq V$ denote the list of maximal cliques of $\overline{G}$ (a chordal extension of $G$) and $\B_k=\{\b_i\in\B\mid i\in C_k\}, k=1,2,\ldots,t$. Then, $f(\x)\in\Sigma_{\A}$ if and only if there exist $f_k(\x)=(\x^{\B_k})^TQ_k\x^{\B_k}$ with $Q_k\in\mathbf{S}_+^{|C_k|}$ for $k=1,\ldots,t$ such that
\begin{equation}\label{sec3-eq9}
f(\x)=\sum_{k=1}^tf_k(\x).
\end{equation}
\end{theorem}

By virtue of Theorem \ref{sec3-thm1}, checking membership in $\Sigma_{\A}$ boils down to solving an SDP problem involving PSD matrices of small sizes if each maximal clique of $\overline{G}$ has a small size relative to the original matrix. This might significantly reduce the overall computational cost.

\section{The SparseJSR Algorithm}
\label{sec:sparsejsr}
In this section, we propose an algorithm for bounding JSR based on the sparse SOS decomposition discussed in the previous section. To this end, we first establish a hierarchy of sparse supports for the auxiliary form $p(\x)$ used in the SOS program \eqref{densesos}.

Let $\AA=\{A_1,\ldots,A_m\}\subseteq\R^{n\times n}$ be a tuple of matrices. Fixing a relaxation order $d$, let $p_0(\x)=\sum_{j=1}^nc_jx_j^{2d}$ with random coefficients $c_{j}\in (0,1)$ and let $\A^{(0)}=\supp(p_0)$. Then for $s\in\N\backslash\{0\}$, we iteratively define
\begin{equation}\label{supp}
    \A^{(s)}:=\A^{(s-1)}\cup\bigcup_{i=1}^m\supp(p_{s-1}(A_i\x)),
\end{equation}
where $p_{s-1}(\x)=\sum_{\a\in\A^{(s-1)}}c_{\a}\x^{\a}$ with random coefficients $c_{\a}\in (0,1)$. Note that the particular form of $p_0(\x)$ is chosen such that $\A^{(s)}$ contains all possible homogeneous monomials of degree $2d$ that are ``compatible" with the couplings between variables $x_1,\ldots, x_n$ introduced by the mappings $\x\mapsto A_i\x$ for all $i$. It is clear that
\begin{equation}\label{hsupp}
    \A^{(1)}\subseteq\cdots\subseteq\A^{(s)}\subseteq\A^{(s+1)}\subseteq\cdots\subseteq\N^n_{2d}
\end{equation}
and the sequence $\{\A^{(s)}\}_{s\ge1}$ stabilizes in finitely many steps. We point out that it is not guaranteed a hierarchy of sparse supports is always retrieved in \eqref{hsupp} even if all $A_i$ are sparse. For instance,
if some matrix $A_i\in\AA$ has a fully dense row, then by definition, one immediately has $\A^{(1)}=\N^n_{2d}$. In this case, the sparsity of $\AA$ cannot be exploited by the present method. This obstacle might be overcome if a more suitable $p_0(\x)$ is chosen taking into account the sparsity pattern of $\AA$, which we leave for future investigation. 

On the other hand, if the matrices in $\AA$ have some common zero columns, then a hierarchy of sparse supports must be retrieved.

\begin{proposition}\label{prop}
Let $\AA=\{A_1,\ldots,A_m\}\subseteq\R^{n\times n}$ and assume that the matrices in $\AA$ have common zero columns indexed by $J\subseteq[n]:=\{1,2,\ldots,n\}$. Let $\tilde{\N}^{n-|J|}_{2d}:=\{(\alpha_i)_{i\in[n]}\in\N^n\mid(\alpha_i)_{i\in[n]\backslash J}\in\N^{n-|J|}_{2d},\alpha_i=0\textrm{ for }i\in J\}$ and $\mathbf{b}_j:=\{(\alpha_i)_{i\in[n]}\in\N^n\mid\alpha_j=2d,\alpha_i=0\textrm{ for }i\ne j\}$ for $j\in[n]$. Then $\A^{(s)}\subseteq\tilde{\N}^{n-|J|}_{2d}\cup\{\mathbf{b}_j\}_{j\in J}$ for all $s\ge1$. 
\end{proposition}
\begin{proof}
Let us do induction on $s$. It is obvious that $\A^{(0)}\subseteq\tilde{\N}^{n-|J|}_{2d}\cup\{\mathbf{b}_j\}_{j\in J}$.
Now assume $\A^{(s)}\subseteq\tilde{\N}^{n-|J|}_{2d}\cup\{\mathbf{b}_j\}_{j\in J}$ for some $s\ge0$. Since the variables effectively involved in $p_s(A_j\x)$ are contained in $\{x_i\}_{i\in[n]\backslash J}$, we have $\supp(p_s(A_j\x))\subseteq\tilde{\N}^{n-|J|}_{2d}$ for $j=1,\ldots,m$. This combined with the induction hypothesis yields $\A^{(s+1)}\subseteq\tilde{\N}^{n-|J|}_{2d}\cup\{\mathbf{b}_j\}_{j\in J}$ as desired.
\end{proof}

For each $s\ge1$, by restricting $p(\x)$ to forms with the sparse support $\A^{(s)}$, \eqref{densesos} now reads as
\begin{align}\label{ssos}
    &\inf_{p\in\R[\A^{(s)}],\gamma}\gamma\\
    &\textrm{ s.t. }\begin{cases}p(\x)-||\x||_2^{2d}\in\Sigma_{n,2d},\\
    \gamma^{2d}p(\x)-p(A_i\x)\in\Sigma_{n,2d},\ 1 \leq i \leq m. \nonumber
    \end{cases}
\end{align}

Let $\A_i^{(s)}=\A^{(s)}\cup\supp(p_s(A_i\x)$ for $i=1,\ldots,m$. In order to exploit the sparsity present in \eqref{ssos}, we then replace $\Sigma_{n,2d}$ with $\Sigma_{\A^{(s)}}$ or $\Sigma_{\A_i^{(s)}}$ in \eqref{ssos}. Consequently we obtain a hierarchy of SOS relaxations indexed by $s$ for a fixed $d$:
\begin{align}\label{tssos}
    \rho_{s,2d}(\AA)& :=\inf_{p\in\R[\A^{(s)}],\gamma}\gamma \\
    & \quad\textrm{ s.t. }\begin{cases}p(\x)-||\x||_2^{2d}\in\Sigma_{\A^{(s)}},\\
    \gamma^{2d}p(\x)-p(A_i\x)\in\Sigma_{\A^{(s)}_i},\ 1 \leq i \leq m. \nonumber
    \end{cases}
\end{align}
We call the index $s$ the {\em sparse order} of \eqref{tssos}. As in the dense case, the optimization problem \eqref{tssos} can be solved via SDP by bisection on $\gamma$.
Moreover, we have the following theorem.
\begin{theorem}\label{thm2}
Let $\AA=\{A_1,\ldots,A_m\}\subseteq\R^{n\times n}$. For any integer $d\ge1$, one has $\rho_{\textrm{SOS},2d}(\AA)\le\cdots\le\rho_{s,2d}(\AA)\le\cdots\le\rho_{2,2d}(\AA)\le\rho_{1,2d}(\AA)$.
\end{theorem}
\begin{proof}
For any fixed $d\in\N\backslash\{0\}$, because of \eqref{hsupp}, it is clear that the feasible set of \eqref{tssos} with the sparse order $s$ is contained in the feasible set of \eqref{tssos} with the sparse order $s+1$, which is in turn contained in the feasible set of \eqref{densesos}. This yields the desired conclusion.
\end{proof}

So we can propose the algorithm {\tt SparseJSR} that computes a non-increasing sequence of upper bounds for the JSR of a tuple of matrices via solving \eqref{tssos} for any fixed $d$. By varying the relaxation order $d$ and the sparse order $s$, {\tt SparseJSR} offers a trade-off between the computational cost and the quality of the obtained upper bound. The correctness of {\tt SparseJSR} is guaranteed by Theorem \ref{sec2-thm1} and Theorem \ref{thm2}.



\section{Numerical Experiments}
\label{sec:benchs}
In this section, we present numerical experiments for the proposed algorithm {\tt SparseJSR}, which is implemented in the Julia package also named \href{https://github.com/wangjie212/SparseJSR}{\tt SparseJSR} and based on the \href{https://github.com/wangjie212/TSSOS}{\tt TSSOS} package used in \cite{wang2,wang3,wang4}. {\tt SparseJSR} utilizes the Julia packages {\tt LightGraphs} \cite{graph} to handle graphs, {\tt ChordalGraph} \cite{Wang20} to generate chordal extensions and {\tt JuMP} \cite{jump} to model SDP. 
Finally, {\tt SparseJSR} relies on the SDP solver {\tt MOSEK} \cite{mosek} to solve SDP.
For the comparison purpose, we also implement the dense SOS relaxation \eqref{densesos} in {\tt SparseJSR} using the same SDP solver {\tt MOSEK}. 
%
For all examples, the sparse order $s$ is set as $1$, the tolerance for bisection is set as $\epsilon=1\times10^{-5}$, and the initial interval for bisection is set as $[0,2]$. To measure the quality of upper bounds that we obtain, a lower bound for JSR is also computed using the MATLAB {\tt JSR} toolbox \cite{jsrtoolbox}. All examples were computed on an Intel Core i5-8265U@1.60GHz CPU with 8GB RAM memory.
The notations that we use are listed in Table \ref{table1}.
\begin{table}[htbp]
\caption{The notations}\label{table1}
\begin{center}
\begin{tabular}{|c|c|}
\hline
$m$&the number of matrices in $\AA$\\
\hline
$n$&the size of matrices in $\AA$\\
\hline
$lb$&lower bounds for JSR given by the {\tt JSR} toolbox\\
\hline
$ub$&upper bounds for JSR given by {\tt SparseJSR}\\
\hline
$d$&the relaxation order\\
\hline
$mb$&the maximal size of PSD blocks\\
\hline
time&running time in seconds\\
\hline
-&$>3600\,$s\\
\hline
$*$&an out of memory error\\
\hline
\end{tabular}
\end{center}
\end{table}

We consider randomly generated examples and examples arising from the study of deadline hit/miss in \cite{maggio2020control}.

\subsection{Randomly generated examples}
We generate random sparse matrices as follows\footnote{Available at https://wangjie212.github.io/jiewang/code.html.}: first call the function ``erdos\_renyi" in the Julia packages {\tt LightGraphs} to generate a random directed graph $G$ with $n$ nodes and $n+10$ edges; for each edge $(i,j)$ of $G$, put a random number in $[-1,1]$ on the position $(i,j)$ of the matrix and put zeros for other positions. We compute an upper bound of the JSR for pairs of such matrices with different sparsity patterns using the first-order SOS relaxations. The results are displayed in Table \ref{random}. It is evident that the sparse approach is much more efficient than the dense approach. 
For instance, the dense approach takes over $3600\,$s when the size of matrices is greater than $100$ while the sparse approach can easily handle matrices of size $120$ within $12\,$s. Both the dense approach and the sparse approach produce upper bounds which are within $0.05$ greater than the corresponding lower bounds.


\begin{table}[htbp]
\caption{Randomly generated examples with $d=1$ and $m = 2$}\label{random}
\rowcolors{2}{white}{gray!25}
\begin{center}
\begin{tabular}{cccccccc}
\hline
\rowcolor{gray!50}
&&\multicolumn{3}{c}{Sparse ($d=1$)}&\multicolumn{3}{c}{Dense ($d=1$)}\\
\rowcolor{gray!50}
\multirow{-2}*{$n$}&\multirow{-2}*{$lb$}&time&$ub$&$mb$&time&$ub$&$mb$\\
\hline
$20$&$0.7894$&$0.74$&$0.8192$&$10$&$1.88$&$0.7967$&$20$\\
\hline
$30$&$0.8502$&$1.65$&$0.8666$&$10$&$7.79$&$0.8523$&$30$\\
\hline
$40$&$0.9446$&$2.68$&$0.9446$&$14$&$25.6$&$0.9446$&$40$\\
\hline
$50$&$0.8838$&$2.97$&$0.9102$&$14$&$55.9$&$0.8838$&$50$\\
\hline
$60$&$0.7612$&$3.64$&$0.7843$&$13$&$171$&$0.7612$&$60$\\
\hline
$70$&$0.9629$&$4.35$&$0.9629$&$11$&$308$&$0.9629$&$70$\\
\hline
$80$&$0.9345$&$5.95$&$0.9399$&$15$&$743$&$0.9345$&$80$\\
\hline
$90$&$0.8020$&$6.27$&$0.8465$&$14$&$1282$&$0.8020$&$90$\\
\hline
$100$&$0.8642$&$8.15$&$0.9132$&$13$&$2568$&$0.8659$&$100$\\
\hline
$110$&$0.8355$&$9.59$&$0.8839$&$15$&-&-&-\\
\hline
$120$&$0.7483$&$11.7$&$0.7735$&$16$&-&-&-\\
\hline
\end{tabular}
\end{center}
\end{table}

\subsection{Examples from control systems}

Here we consider examples from \cite{maggio2020control}, where the dynamics of closed-loop systems are given {by the combination of a plant and a one-step delay controller that stabilizes the plant. The closed-loop system evolves according to either a completed or a missed computation. In the case of a deadline hit, the closed-loop state matrix is $A_H$. In the case of a deadline miss, the associated closed-loop state matrix is $A_M$. The computational platform (hardware and software) ensures that no more than $m-1$ deadlines are missed consecutively.}
The set of possible realisations $\AA$ of such a system contains either a single hit or at most $m-1$ misses followed by a hit, namely $\AA := \{A_H A_M^i \mid 0 \leq  i \leq m-1 \}$.
Then, the closed-loop system that can switch between the realisations included in $\AA$  is asymptotically stable if and only if $\rho (\AA) < 1$. {This gives an indication for scheduling and control co-design, in which the hardware and software platform must guarantee that the maximum number of deadlines missed consecutively does not interfere with stability requirements.}

In Table \ref{control1} and Table \ref{control2}, we report the results obtained for various control systems with $n$ states, under $m-1$ deadline misses, by applying the dense and sparse relaxations with relaxation orders $d=1$ and $d=2$, respectively.
The examples are randomly generated, i.e., our script generates a
random system and then tries to control it\footnote{Available at https://wangjie212.github.io/jiewang/code.html.}. 

In Table \ref{control1}, we fix $m=5$ and vary $n$ from $20$ to $110$. For these examples, surprisingly the dense and sparse approaches with the relaxation order $d=1$ always produce the same upper bounds. As we can see from the table, the sparse approach is more scalable and efficient than the dense one.

In Table \ref{control2}, we vary $m$ from $2$ to $11$ and vary $n$ from $6$ to $24$. For each instance, one has $mb = 10$ for the sparse approach.
The column ``$ub$'' indicates the upper bound given by the dense approach with the relaxation order $d=1$. For these examples, with the relaxation order $d=2$, the sparse approach produces upper bounds that are very close to those given by the dense approach. And again the sparse approach is more scalable and more efficient than the dense one.

\begin{table}[htbp]
\caption{Results for control systems with $d=1$ and $m = 5$}\label{control1}
\rowcolors{2}{white}{gray!25}
\begin{center}
\begin{tabular}{cccccccc}
\hline
\rowcolor{gray!50}
&&\multicolumn{3}{c}{Sparse ($d=1$)}&\multicolumn{3}{c}{Dense ($d=1$)}\\
\rowcolor{gray!50}
\multirow{-2}*{$n$}&\multirow{-2}*{$lb$}&time&$ub$&$mb$&time&$ub$&$mb$\\
\hline
$20$&$0.9058$&$1.78$&$0.9316$&$12$&$9.92$&$0.9316$&$20$\\
\hline
$20$&$0.8142$&$1.62$&$0.8142$&$12$&$9.08$&$0.8142$&$20$\\
\hline
$30$&$1.4682$&$4.30$&$1.5132$&$14$&$57.8$&$1.5131$&$30$\\
\hline
$30$&$1.0924$&$4.42$&$1.0961$&$14$&$65.4$&$1.0961$&$30$\\
\hline
$40$&$1.1648$&$9.29$&$1.1977$&$16$&$249$&$1.1977$&$30$\\
\hline
$40$&$0.9772$&$9.69$&$0.9804$&$16$&$259$&$0.9804$&$30$\\
\hline
$50$&$1.3153$&$17.3$&$1.3248$&$18$&$660$&$1.3248$&$50$\\
\hline
$50$&$1.1884$&$17.5$&$1.1884$&$18$&$680$&$1.1884$&$50$\\
\hline
$60$&$1.8366$&$29.7$&$1.8820$&$20$&$2049$&$1.8820$&$60$\\
\hline
$60$&$1.3259$&$30.7$&$1.3259$&$20$&$1776$&$1.3259$&$60$\\
\hline
$70$&$1.8135$&$54.2$&$1.8578$&$22$&-&-&-\\
\hline
$70$&$1.2727$&$53.9$&$1.2727$&$22$&-&-&-\\
\hline
$80$&$2.3005$&$85.3$&$2.3445$&$24$&-&-&-\\
\hline
$80$&$1.4262$&$85.6$&$1.4262$&$24$&-&-&-\\
\hline
$90$&$1.8745$&$133$&$1.9020$&$26$&-&-&-\\
\hline
$90$&$1.4452$&$132$&$1.4452$&$26$&-&-&-\\
\hline
$100$&$2.2316$&$196$&$2.2733$&$28$&$*$&$*$&$*$\\
\hline
$100$&$1.5267$&$195$&$1.5267$&$28$&$*$&$*$&$*$\\
\hline
$110$&$2.3597$&$280$&$2.3943$&$30$&$*$&$*$&$*$\\
\hline
$110$&$1.5753$&$287$&$1.5753$&$30$&$*$&$*$&$*$\\
\hline
\end{tabular}
\end{center}
\end{table}

\begin{table}[htbp]
\caption{Results for control systems with $d=2$}\label{control2}
\rowcolors{2}{white}{gray!25}
\begin{center}
\begin{tabular}{p{0.15cm}p{0.15cm}p{0.7cm}p{0.7cm}p{0.4cm}p{1cm}p{0.4cm}p{0.7cm}c}
\hline
\rowcolor{gray!50}
&&&&\multicolumn{2}{c}{Sparse ($d=2$)}&\multicolumn{3}{c}{Dense ($d=2$)}\\
\rowcolor{gray!50}
\multirow{-2}*{$m$}&\multirow{-2}*{$n$}&\multirow{-2}*{$\quad lb$}&\multirow{-2}*{$\quad ub$}&time&\multicolumn{1}{c}{$ub$}&time&\multicolumn{1}{c}{$ub$}&$mb$\\
\hline
$2$&$6$&$0.9464$&$0.9782$&$0.42$&$0.9547$&$1.87$&$0.9539$&$21$\\
\hline
$3$&$8$&$0.7218$&$0.7467$&$0.60$&$0.7310$&$13.4$&$0.7305$&$36$\\
\hline
$4$&$10$&$0.7458$&$0.7738$&$0.75$&$0.7564$&$107$&$0.7554$&$55$\\
\hline
$5$&$12$&$0.8601$&$0.8937$&$1.08$&$0.8706$&$1157$&$0.8699$&$78$\\
\hline
$6$&$14$&$0.7875$&$0.8107$&$1.32$&$0.7958$&-&-&-\\
\hline
$7$&$16$&$1.1110$&$1.1531$&$1.81$&$1.1182$&$*$&$*$&$*$\\
\hline
$8$&$18$&$1.0487$&$1.0881$&$2.05$&$1.0569$&$*$&$*$&$*$\\
\hline
$9$&$20$&$0.7570$&$0.7808$&$2.52$&$0.7660$&$*$&$*$&$*$\\
\hline
$10$&$22$&$0.9911$&$1.0315$&$2.70$&$1.0002$&$*$&$*$&$*$\\
\hline
$11$&$24$&$0.7339$&$0.7530$&$3.67$&$0.7418$&$*$&$*$&$*$\\
\hline
\end{tabular}
\end{center}
\end{table}



\end{document}